\documentclass[12pt]{article}
\usepackage{graphicx}
\usepackage{amsmath,amsthm,amssymb,amsfonts,euscript,enumerate}
\usepackage{amsthm}
\usepackage{color,wrapfig}
\usepackage{pxfonts}
\usepackage{verbatim}
\newtheorem{thm}{Theorem}[section]
\newtheorem{cor}[thm]{Corollary}
\newtheorem{lem}[thm]{Lemma}

\newcommand{\R}{{\mathbb{R}}}
\newcommand{\Z}{{\mathbb{Z}}}

\newcommand{\1}{\partial}
\newcommand{\2}{\overline}
\newcommand{\3}{\varepsilon}
\newcommand{\4}{\widetilde}

\def\ni{\noindent}

\begin{document}
\title{Asymptotic behaviour of solutions of the fast diffusion equation near its extinction time}
\author{Kin Ming Hui\\
%\thanks{ }\\
Institute of Mathematics, Academia Sinica\\
Taipei, Taiwan, R. O. C.}
\date{Nov 15, 2014}
\smallbreak \maketitle
\begin{abstract}
Let $n\ge 3$, $0<m<\frac{n-2}{n}$, $\rho_1>0$, $\beta\ge\frac{m\rho_1}{n-2-nm}$ and $\alpha=\frac{2\beta+\rho_1}{1-m}$. For any $\lambda>0$, we will prove the existence and uniqueness (for $\beta\ge\frac{\rho_1}{n-2-nm}$) of radially symmetric singular solution $g_{\lambda}\in C^{\infty}(\R^n\setminus\{0\})$ of the  elliptic equation $\Delta v^m+\alpha v+\beta x\cdot\nabla v=0$, $v>0$, in $\R^n\setminus\{0\}$, satisfying $\displaystyle\lim_{|x|\to 0}|x|^{\alpha/\beta}g_{\lambda}(x)=\lambda^{-\frac{\rho_1}{(1-m)\beta}}$. When  $\beta$ is sufficiently large, we prove the higher order asymptotic behaviour of radially symmetric solutions of the above elliptic equation as $|x|\to\infty$. We also obtain an inversion formula for the radially symmetric solution of the above equation. As a consequence 
we will prove the extinction behaviour of the solution $u$ of the fast diffusion equation $u_t=\Delta u^m$ in $\R^n\times (0,T)$ near the extinction time $T>0$. 
\end{abstract}

\vskip 0.2truein



\vskip 0.2truein
\setcounter{equation}{0}
\setcounter{section}{0}

\section{Introduction}
\setcounter{equation}{0}
\setcounter{thm}{0}

The equation 
\begin{equation}\label{fde}
u_t=\Delta u^m
\end{equation}
appears in many physical models. When $m>1$, \eqref{fde} is the porous medium equation which models the flow of gases or liquid through porous media. When $m=1$, \eqref{fde} is the heat equation. When $0<m<1$, \eqref{fde} is the fast diffusion equation. When $m=\frac{n-2}{n+2}$, $n\ge 3$, and $g=u^{\frac{4}{n+2}}dx^2$ is a metric on $\R^n$ which evolves by the Yamabe flow,
\begin{equation*}
\frac{\1 g}{\1 t}=-Rg
\end{equation*}
where $R$ is the scalar curvature of the metric $g$, then $u$ satisfies \cite{DKS}, \cite{PS},
\begin{equation*}
u_t=\frac{n-2}{m}\Delta u^m.
\end{equation*}
It is because of the importance of the equation \eqref{fde} and its relation to the Yamabe flow, there are a lot of research on this equation recently by P.~Daskalopoulos, J.~King, M.~del Pino, N.~Sesum, M.~S\'aez,
\cite{DKS}, \cite{DPS}, \cite{DS1}, \cite{DS2}, \cite{PS}, S.Y.~Hsu [Hs1--3], K.M.~Hui \cite{Hu1}, \cite{Hu2}, M.~Fila, J.L.~Vazquez, M.~Winkler, E.~Yanagida, \cite{FVWY}, \cite{FW}, A.~Blanchet, M.~Bonforte, J.~Dolbeault, G.~Grillo and J.L.~Vazquez, \cite{BBDGV}, \cite{BDGV}, etc. We refer the reader to the survey paper \cite{A} by D.G.~Aronson and the books \cite{DK}, \cite{V2}, by P.~Daskalopoulos,  C.E.~Kenig, and J.L.~Vazquez on the recent progress on this equation.

As observed by J.L.~Vazquez \cite{V1}, M.A.~Herrero and M.~Pierre \cite{HP}, and others \cite{Hs2}, \cite{Hu1}, there is a big difference on the behaviour of solution of \eqref{fde} for the case $\frac{n-2}{n}<m<1$, $n\ge 3$, and the case $0<m\le\frac{n-2}{n}$, $n\ge 3$. For example for any $0\le u_0\in L_{loc}^1(\R^n)$, $u_0\not\equiv 0$, when $\frac{n-2}{n}<m<1$, $n\ge 3$, there exists a unique global positive smooth solution of \eqref{fde} in $\R^n\times (0,\infty)$ with initial value $u_0$ \cite{HP}. 
On the other hand when $0<m<\frac{n-2}{n}$, $n\ge 3$, there exists $0\le u_0\in L_{loc}^1(\R^n)$, $u_0\not\equiv 0$, and $T>0$ such that the solution of 
\begin{equation}\label{Cauchy-problem}
\left\{\begin{aligned}
u_t=&\Delta u^m \quad\mbox{ in }\R^n\times (0,T)\\
u(x,0)=&u_0\qquad\mbox { in }\R^n.
\end{aligned}\right.
\end{equation}
extincts at time $T$ \cite{DS1}. Since the asymptotic behaviour of the solution of \eqref{Cauchy-problem}
near the extinction time is usually similar to the asymptotic behaviour of the self-similar solution of \eqref{fde}, in order to understand the behaviour of the solution of \eqref{Cauchy-problem} near the extinction time
we will first study various properties of the self-similar solutions of \eqref{fde} in this paper. 

Let $n\ge 3$, $0<m<\frac{n-2}{n}$, $\rho_1>0$, $\beta>\frac{m\rho_1}{n-2-nm}$ and $\alpha=\frac{2\beta+\rho_1}{1-m}$. For any $\lambda>0$, by Theorem 1.1 of \cite{Hs1} there exists a unique radially symmetric solution  $v_{\lambda}$ 
of the equation
\begin{equation}\label{elliptic-eqn}
\Delta v^m+\alpha v+\beta x\cdot\nabla v=0, v>0,
\end{equation}
in $\R^n$ that satisfies $v_{\lambda}(0)=\lambda$. 
By \cite{Hs3}, $v_{\lambda}$ satisfies
\begin{equation}\label{v-infty-behaviour}
\lim_{r\to\infty}r^2v_{\lambda}(r)^{1-m}=\frac{2m(n-2-nm)}{(1-m)\rho_1}.
\end{equation}
Note that when $\rho_1=1$, the function
\begin{equation}\label{self-similar-soln}
\psi_{\lambda}(x,t)=(T-t)^{\alpha}v_{\lambda}((T-t)^{\beta}x)
\end{equation} 
is a solution of \eqref{fde} in $\R^n\times (0,T)$ for any $T>0$. On the other hand if $\rho_1=1$, $m=\frac{n-2}{n+2}$, and $n\ge 3$, then the metric
\begin{equation}\label{g-metric}
g=\left(\left(\frac{n-1}{m}\right)^{\frac{1}{1-m}}v_{\lambda}\right)^{\frac{4}{n+2}}dx^2
\end{equation}
on $\R^n$ is a Yamabe shrinking soliton \cite{DS2}. Conversely as proved by P.~Daskalopoulos and N.~Sesum \cite{DS2} any Yamabe shrinking soliton on complete locally conformally flat manifold is of the form \eqref{g-metric}  where $v_{\lambda}$ is a solution of \eqref{elliptic-eqn} in $\R^n$ for some  $\alpha=\frac{2\beta+1}{1-m}$ with $v_{\lambda}(0)=\lambda$ for some constant $\lambda>0$. 

Let $\beta_1=\frac{\rho_1}{n-2-nm}$,
\begin{equation*}
\beta_0=\left\{\begin{aligned}
&\rho_1\sqrt{\frac{2(1-m)}{n-2-nm}}\qquad\qquad\qquad\qquad\qquad\qquad\quad\mbox{ if }0<m\le\frac{n-2}{n+2}\\
&\rho_1\max\left(2\sqrt{\frac{2(1-m)}{n-2-nm}},\frac{(n+2)m-(n-2)}{n-2-nm}\right)\quad\mbox{ if }\,\frac{n-2}{n+2}<m<\frac{n-2}{n},\end{aligned}\right. 
\end{equation*} 
and $\gamma_2$, $\gamma_1$, be the two roots of the equation
\begin{equation}\label{gamma-eqn}
\gamma^2-\left(\frac{n-2-(n+2)m}{1-m}+\frac{2\beta(n-2-nm)}{(1-m)\rho_1}\right)\gamma+\frac{2(n-2-nm)}{(1-m)}=0
\end{equation}
given by 
\begin{equation}\label{roots}
\gamma_i=\frac{1}{2(1-m)}\left\{A(\beta)+(-1)^i\sqrt{A(\beta)^2-8(n-2-nm)(1-m)}\right\},\quad i=1,2
\end{equation}
where
\begin{equation*}
A(\beta)=n-2-(n+2)m+\frac{2\beta(n-2-nm)}{\rho_1}.
\end{equation*}
Now if $\beta\ge\beta_0$, then
\begin{align*}
&A(\beta)^2-8(n-2-nm)(1-m)\\
\ge&\left\{\begin{aligned}
&\frac{4(n-2-nm)^2}{\rho_1^2}\left\{\beta^2-\frac{2(1-m)\rho_1^2}{n-2-nm}\right\}\quad\mbox{ if }0<m\le\frac{n-2}{n+2}\\
&\frac{(n-2-nm)^2}{\rho_1^2}\left\{\beta^2-\frac{8(1-m)\rho_1^2}{n-2-nm}\right\}\quad\mbox{ if }\frac{n-2}{n+2}<m<\frac{n-2}{n}\end{aligned}\right.\\
\ge&0.
\end{align*}
Hence $\gamma_2\ge\gamma_1>0$ are real roots of \eqref{gamma-eqn} when $0<m<\frac{n-2}{n}$, $n\ge 3$, and $\beta\ge\beta_0$. Note that 
\begin{equation}\label{beta>beta1-cond}
\beta>\beta_1\quad (\beta=\beta_1) \Leftrightarrow\quad n\beta>\alpha\quad (n\beta=\alpha\mbox{ respectively}).
\end{equation}
and when $m=\frac{n-2}{n+2}$ and $\rho_1=1$, then $\beta_0=\frac{2}{\sqrt{n-2}}$, $\beta_1=\frac{1}{2m}$, and \eqref{gamma-eqn} is equivalent to 
\begin{equation*}
\gamma^2-\beta (n-2)\gamma+(n-2)=0.
\end{equation*}
In \cite{DKS} P.~Daskalopoulos, J.~King and N.~Sesum, proved that when $m=\frac{n-2}{n+2}$, $n\ge 3$,
$\rho_1=1$, and $\beta>\beta_0$, the radially symmetric solution $v_{\lambda}$ of \eqref{elliptic-eqn} in $\R^n$ with $v_{\lambda}(0)=\lambda$ satisfies
\begin{equation}\label{v-lambda-infty-behaviour}
v_{\lambda}(x)=\left(\frac{C_{\ast}}{|x|^2}\right)^{\frac{1}{1-m}}(1-B|x|^{-\gamma}+o(|x|^{-\gamma}))\quad\mbox{ as }|x|\to\infty
\end{equation} 
for some constants $B\in\R$, 
\begin{equation}\label{c-ast-defn}
C_{\ast}=\frac{2m(n-2-nm)}{(1-m)\rho_1},
\end{equation}
and $\gamma>0$ where $\gamma=\gamma_2$ if $3\le n<6$ and $\beta=\beta_1$, and $\gamma=\gamma_1$ otherwise. In this paper we will extend this second order asymptotic result to the case $0<m<\frac{n-2}{n}$, $n\ge 3$ and $\rho_1>0$. For any 
$0<m<\frac{n-2}{n}$, $n\ge 3$, and $\lambda>0$, we will also extend Theorem 1.2 of \cite{DKS} and prove the existence and uniqueness of radially symmetric singular solution $g_{\lambda}\in C^{\infty}(\R^n\setminus\{0\})$ of \eqref{elliptic-eqn} in $\R^n\setminus\{0\}$ that satisfies
\begin{equation}\label{blow-up-rate-x=0}
\lim_{|x|\to 0}|x|^{\alpha/\beta}g_{\lambda}(x)=\lambda^{-\frac{\rho_1}{(1-m)\beta}}.
\end{equation} 
We also obtain higher order decay rate of $g_{\lambda}$ as $|x|\to\infty$.
Let
\begin{equation*}
\mathcal{C}(x)=\left(\frac{C_{\ast}}{r^2}\right)^{\frac{1}{1-m}}.
\end{equation*}
Then $\mathcal{C}(x)$ is a solution of \eqref{elliptic-eqn} in $\R^n\setminus\{0\}$.
In the papers \cite{DS1}, \cite{BBDGV}, etc. P.~Daskalopoulos and N.~Sesum, 
A.~Blanchet, M.~Bonforte, J.~Dolbeault, G.~Grillo and J.L.~Vazquez, etc. obtain the asymptotic behaviour of the solution of \eqref{Cauchy-problem} near the extinction time for $0<m<\frac{n-2}{n}$, $n\ge 3$, when the initial value is sandwiched between
two Barenblatt solutions. In this paper we will extend their results to initial values that satisfies other growth conditions.

More precisely we obtain the following main results in this paper.

\begin{thm}\label{blow-up-self-similar-soln-thm}
Let $n\ge 3$, $0<m<\frac{n-2}{n}$, $\rho_1>0$, $\lambda>0$, $\beta\ge\frac{m\rho_1}{n-2-nm}$ and $\alpha=\frac{2\beta+\rho_1}{1-m}$. Then there exists a  radially symmetric solution $g_{\lambda}$ of \eqref{elliptic-eqn} in $\R^n\setminus\{0\}$ that satisfies \eqref{blow-up-rate-x=0},
\begin{equation}\label{g-lambda'-negative}
g_{\lambda}'(r)\le 0\quad\forall r>0
\end{equation}
and
\begin{equation}\label{g-lower-bd}
\lambda^{-\frac{\rho_1}{(1-m)\beta}}\le r^{\frac{\alpha}{\beta}}g_{\lambda}(r)\le\lambda^{-\frac{\rho_1}{(1-m)\beta}}\mbox{exp}\,\left(\frac{\beta C_0}{\rho_1}\lambda^{\frac{\rho_1}{\beta}}r^{\frac{\rho_1}{\beta}}\right)\quad\forall r>0.
\end{equation}
for some constant $C_0>0$. Moreover if $\beta\ge\beta_1$, then the solution is unique.
\end{thm}

\begin{cor}\label{V-cor}
Let $n\ge 3$, $0<m<\frac{n-2}{n}$, $\rho_1=1$, $\lambda>0$, $\beta\ge\frac{m}{n-2-nm}$ and $\alpha=\frac{2\beta+1}{1-m}$. Let $g_{\lambda}$ be the radially symmetric solution of \eqref{elliptic-eqn} in $\R^n\setminus\{0\}$ that satisfies \eqref{blow-up-rate-x=0}. Then for any $T>0$ the function
\begin{equation}\label{V-lambda-defn}
V_{\lambda}(x,t)=(T-t)^{\alpha}g_{\lambda}((T-t)^{\beta}x)
\end{equation}
satisfies \eqref{fde} in $\R^n\setminus\{0\}$ and 
\begin{equation*}
\lim_{|x|\to 0}|x|^{\frac{\alpha}{\beta}}V_{\lambda}(x,t)=\lambda^{-\frac{1}{(1-m)\beta}}.
\end{equation*}
\end{cor}

\begin{thm}\label{second-order-asymptotic-self-similar-soln-thm}
Let $n\ge 3$, $0<m<\frac{n-2}{n}$, $\rho_1>0$, and $\lambda>0$. Then there exists a constant $\beta_2\ge\max\left(\beta_0, \beta_1\right)$ depending on $n$ and $m$ such that for any
$\beta>\beta_2$, $\alpha=\frac{2\beta+\rho_1}{1-m}$, if $v_{\lambda}$ is a radially symmetric solution of \eqref{elliptic-eqn} in $\R^n$ with $v_{\lambda}(0)=\lambda$, then \eqref{v-lambda-infty-behaviour} holds for some constants $B>0$ and $\gamma=\gamma_1>0$
where $\gamma_1$ is given by \eqref{roots}.
\end{thm}

\begin{thm}\label{second-order-asymptotic-blow-up-elliptic-soln-thm}
Let $n\ge 3$, $0<m<\frac{n-2}{n}$, $\rho_1>0$, $\lambda>0$  and let $\beta_2>0$ be as in Theorem \ref{second-order-asymptotic-self-similar-soln-thm}. Then for any
$\beta>\beta_2$, $\alpha=\frac{2\beta+\rho_1}{1-m}$, if $g_{\lambda}$ is the unique radially symmetric solution of \eqref{elliptic-eqn} in $\R^n\setminus\{0\}$ which satisfies \eqref{blow-up-rate-x=0} that is given by Theorem \ref{blow-up-self-similar-soln-thm}, then 
\begin{equation}\label{g-lambda-infty-behaviour0}
g_{\lambda}(x)=\left(\frac{C_{\ast}}{|x|^2}\right)^{\frac{1}{1-m}}(1+B|x|^{-\gamma_1}+o(|x|^{-\gamma_1}))\quad\mbox{ as }|x|\to\infty
\end{equation} 
holds for some constants $B>0$ where $\gamma_1$ is given by \eqref{roots}.
\end{thm}

By direct computation we also have the following inversion formula for the solution of \eqref{elliptic-eqn}.

\begin{thm}\label{inversion-thm}
Let $n\ge 3$, $m=\frac{n-2}{n+2}$, $\alpha$, $\beta\in\R$. Let $v$ be a radially symmetric solution of \eqref{elliptic-eqn} in $\R^n\setminus\{0\}$ and $\4{v}(\rho)=\rho^{-\frac{n-2}{m}}v(\rho^{-1})$. Then $\4{v}$ satisfies 
\begin{equation*}
\Delta\4{v}^m+\alpha'\,\4{v}+\beta'\,x\cdot\nabla\4{v}=0\quad\mbox{ in }\R^n\setminus\{0\}
\end{equation*}
where $\alpha'=\alpha-\frac{n-2}{m}\beta$, $\beta'=-\beta$. If $\alpha=\frac{2\beta+\rho_1}{1-m}$ for some constant $\rho_1>0$, then $\alpha'=\frac{2\beta'+\rho_1}{1-m}$.
Moreover $r^{\frac{\alpha}{\beta}}v(r)=\rho^{\frac{\alpha'}{\beta'}}\4{v}(\rho)$ for all $r=\rho^{-1}>0$.
\end{thm}

For any solution $u$ of \eqref{Cauchy-problem} we let 
\begin{equation}\label{rescald-soln}
\4{u}(x,s)=(T-t)^{-\alpha}u((T-t)^{-\beta}x),\quad s=-\log (T-t).
\end{equation}
Then $\4{u}$ satisfies 
\begin{equation*}\label{rescaled-eqn}
\4{u}_s=\Delta\4{u}^m+\alpha\4{u}+\beta x\cdot\nabla \4{u}\quad\mbox{ in }\R^n\times (-\log T,\infty).
\end{equation*}
\begin{thm}\label{convergence-thm1}
Let $n\ge 3$, $0<m<\frac{n-2}{n}$, $T>0$, $\rho_1=1$, $\beta>\beta_1$ and $\alpha=\frac{2\beta+1}{1-m}$. Let $\psi_{\lambda}$ be given by \eqref{self-similar-soln} and let $u_0$ satisfy $0\le u_0\le\psi_{\lambda_1}(x,0)$ and
\begin{equation*}\label{u0-psi-L1}
|u_0-\psi_{\lambda_0}(x,0)|\le f(|x|)\in L^1(\R^n)
\end{equation*} 
for some constants $\lambda_0>0$, $\lambda_1>0$, and radially symmetric function $f$. Let $u$ be  the  solution of \eqref{Cauchy-problem} and $\4{u}$ be given by \eqref{rescald-soln}.
Then the rescaled solution $\4{u}(\cdot,s)$ converges uniformly on every compact subset of $\R^n$ to
$v_{\lambda_0}$ and in $L^1(\R^n)$ to $v_{\lambda_0}$ as $s\to\infty$.
\end{thm}

\begin{thm}\label{convergence-thm3}
Let $n\ge 3$, $0<m<\frac{n-2}{n}$, $T>0$, $\rho_1=1$, $\beta>\beta_1$ and $\alpha=\frac{2\beta+1}{1-m}$. Let $V_{\lambda}$ be given by \eqref{V-lambda-defn} and let $u_0$ satisfy 
\begin{equation*}
V_{\lambda_1}(x,0)\le u_0(x)\le V_{\lambda_2}(x,0)\quad\mbox{ in }\R^n\setminus\{0\}
\end{equation*}  
and 
\begin{equation*}
u_0-V_{\lambda_0}(x,0)\in L^1(\R^n\setminus\{0\})
\end{equation*} 
for some constants $\lambda_1>\lambda_2>0$ and $\lambda_0>0$. Let $u$ be a solution of \eqref{fde} in 
$(\R^n\setminus\{0\})\times (0,T)$ with initial value $u_0$ which satisfies
\begin{equation*}
V_{\lambda_1}(x,t)\le u(x,t)\le V_{\lambda_2}(x,t)\quad\mbox{ in }(\R^n\setminus\{0\})\times (0,T)
\end{equation*}
and let $\4{u}$ be given by \eqref{rescald-soln}.
Then the rescaled solution $\4{u}(\cdot,s)$ converges uniformly on every compact subset of $\R^n\setminus\{0\}$ and in $L^1(\R^n\setminus\{0\})$ to $g_{\lambda_0}$ as $s\to\infty$. Moreover,
\begin{equation*}
\|\4{u}(\cdot,s)-g_{\lambda_0}\|_{L^1(\R^n\setminus\{0\})}\le e^{-(n\beta-\alpha)s}\|u_0-V_{\lambda_0}(\cdot,0)\|_{L^1(\R^n\setminus\{0\})}\quad\forall s>-\log T.
\end{equation*}
\end{thm}

The plan of the paper is as follows. In section two we will prove Theorem \ref{blow-up-self-similar-soln-thm}. We will prove Theorem \ref{second-order-asymptotic-self-similar-soln-thm} and Theorem \ref{second-order-asymptotic-blow-up-elliptic-soln-thm} in section three and four respectively.
Finally we will sketch the proof of Theorem \ref{convergence-thm1} and Theorem \ref{convergence-thm3} in section five.

Unless stated otherwise we will assume that $n\ge 3$, $0<m<\frac{n-2}{n}$, $\rho_1>0$, $\lambda>0$, $\beta>\frac{m\rho_1}{n-2-nm}$ and $\alpha=\frac{2\beta+\rho_1}{1-m}$, for the rest of the paper. For any $R>0$, we let $B_R=\{x\in\R^n:|x|<R\}$. For any $T>0$ and domain $\Omega\subset\R^n$, we say that $u$ is a solution of \eqref{fde} in $\Omega\times (0,T)$ if $u$ is a smooth positive solution of \eqref{fde} in $\Omega\times (0,T)$. For any $0\le u_0\in L_{loc}^1(\Omega)$ we say that $u$ is a solution of \eqref{fde} in $\Omega\times (0,T)$ with initial value $u_0$ if $u$ is a  solution of \eqref{fde} in $\Omega\times (0,T)$ with $u(\cdot,t)\to u_0$ in $L_{loc}^1(\Omega)$ as $t\to 0$.

\section{Existence of blow-up solutions}
\setcounter{equation}{0}
\setcounter{thm}{0}

In this section we will prove Theorem \ref{blow-up-self-similar-soln-thm}. We first start with a technical lemma.

\begin{lem}\label{g-integral-eqn-lemma}
Let $\beta\ge\beta_1$ and $g_{\lambda}$ be a radially symmetric solution of \eqref{elliptic-eqn} in $\R^n\setminus\{0\}$ which satisfies \eqref{blow-up-rate-x=0}. Then $g_{\lambda}$ satisfies
\begin{equation}\label{v-integral-eqn}
r^{n-1}(v^m)'(r)+\beta r^nv(r)=(n\beta-\alpha)\int_0^rv(\rho)\rho^{n-1}\,d\rho\quad\forall r>0
\end{equation}
if $\beta>\beta_1$ and 
$g_{\lambda}$ satisfies
\begin{equation}\label{v-integral-eqn2}
r^{n-1}(v^m)'(r)+\beta r^nv(r)=\beta\lambda^{-\frac{\rho_1}{(1-m)\beta}}\quad\forall r>0\quad\mbox{ if }\beta=\beta_1.
\end{equation}
\end{lem}
\begin{proof}
We first claim that there exists a sequence of positive numbers $\{\xi_i\}_{i=1}^{\infty}$, $\xi_i\to 0$ as $i\to\infty$, such that 
\begin{equation}\label{g-lambda-dervative-x=0}
\lim_{i\to\infty}\xi_i^{\frac{\alpha}{\beta}+1}g_{\lambda}'(\xi_i)=-\frac{\alpha}{\beta}\lambda^{-\frac{\rho_1}{(1-m)\beta}}.
\end{equation}
In order to proof the claim we choose a sequence of positive numbers $\{r_i\}_{i=1}^{\infty}$ such that $r_i\to 0$ as $i\to\infty$. Then by \eqref{blow-up-rate-x=0} and the mean value theorem for any $i\in\Z^+$ there exists $\xi_i\in (r_i/2,r_i)$ such that
\begin{align*}
&\frac{r_i^{\frac{\alpha}{\beta}+1}g_{\lambda}(r_i)-(r_i/2)^{\frac{\alpha}{\beta}+1}g_{\lambda}(r_i/2)}{r_i/2}=\xi_i^{\frac{\alpha}{\beta}+1}g_{\lambda}'(\xi_i)+\left(1+\frac{\alpha}{\beta}\right)\xi_i^{\frac{\alpha}{\beta}}g_{\lambda}(\xi_i)\\
\Rightarrow\quad&\left|\xi_i^{\frac{\alpha}{\beta}+1}g_{\lambda}'(\xi_i)\right|\le 2r_i^{\frac{\alpha}{\beta}}g_{\lambda}(r_i)+(r_i/2)^{\frac{\alpha}{\beta}}g_{\lambda}(r_i/2)+\left(1+\frac{\alpha}{\beta}\right)\xi_i^{\frac{\alpha}{\beta}}g_{\lambda}(\xi_i)\le C\quad\forall i\in\Z^+
\end{align*}
for some constant $C>0$. Hence the sequence $\{\xi_i\}_{i=1}^{\infty}$ has a subsequece which we may assume without loss of generality to be the sequence itself such that $\xi_i^{\frac{\alpha}{\beta}+1}g_{\lambda}'(\xi_i)$ converges to some constant as $i\to\infty$. Then by \eqref{blow-up-rate-x=0} and  L'Hospital's Rule,
\begin{equation*}
\lambda^{-\frac{\rho_1}{(1-m)\beta}}=\lim_{i\to\infty}\xi_i^{\frac{\alpha}{\beta}}g_{\lambda}(\xi_i)=\lim_{i\to\infty}\frac{g_{\lambda}(\xi_i)}{\xi_i^{-\frac{\alpha}{\beta}}}=\lim_{i\to\infty}\frac{g_{\lambda}'(\xi_i)}{-\frac{\alpha}{\beta}\xi_i^{-\frac{\alpha}{\beta}-1}}=-\frac{\beta}{\alpha}\lim_{i\to\infty}\xi_i^{\frac{\alpha}{\beta}+1}g_{\lambda}'(\xi_i)
\end{equation*}
and \eqref{g-lambda-dervative-x=0} follows. By \eqref{elliptic-eqn} $g_{\lambda}$ satisfies
\begin{align}\label{ode}
&(v^m)''+\frac{n-1}{r}(v^m)'+\alpha v+\beta rv'=0, \,\, v>0\quad\forall r>0\notag\\
\Leftrightarrow\quad&(r^{n-1}(v^m)')'+\alpha r^{n-1}v+\beta r^nv'=0,\,\, v>0\qquad\forall r>0.
\end{align}
Integrating \eqref{ode} over $(\xi,r)$, 
\begin{align}\label{g-lambda-integral-eqn}
&r^{n-1}(g_{\lambda}^m)'(r)+\beta r^ng_{\lambda}(r)\notag\\
=&\xi^{n-1}(g_{\lambda}^m)'(\xi)+\beta\xi^ng_{\lambda}(\xi)+(n\beta-\alpha)\int_{\xi}^rg_{\lambda}(\rho)\rho^{n-1}\,d\rho\quad\forall r>\xi>0\notag\\
=&m\xi^{n-2-\frac{m\alpha}{\beta}}(\xi^{\frac{\alpha}{\beta}}g_{\lambda}(\xi))^{m-1}(\xi^{\frac{\alpha}{\beta}+1}g_{\lambda}'(\xi))+\beta\xi^{\frac{n\beta-\alpha}{\beta}}(\xi^{\frac{\alpha}{\beta}}g_{\lambda}(\xi))+(n\beta-\alpha)\int_{\xi}^rg_{\lambda}(\rho)\rho^{n-1}\,d\rho\quad\forall r>\xi>0
\end{align}
Since $\beta\ge\beta_1$, $n-2-\frac{m\alpha}{\beta}>0$. Putting $\xi=\xi_i$ in \eqref{g-lambda-integral-eqn}
and letting $i\to\infty$, by \eqref{beta>beta1-cond}, \eqref{blow-up-rate-x=0}, and \eqref{g-lambda-dervative-x=0}, we get that $g_{\lambda}$ satisfies \eqref{v-integral-eqn} if $\beta>\beta_1$ and $g_{\lambda}$ satisfies \eqref{v-integral-eqn2} if $\beta=\beta_1$.
\end{proof}

Similarly we have the following lemma.

\begin{lem}\label{v-integral-eqn-lem}
Let $\beta>0$. If $v_{\lambda}$ is a radially symmetric solution of \eqref{elliptic-eqn} in $\R^n$ which satisfies 
$v_{\lambda}(0)=\lambda$, then $v_{\lambda}$ satisfies \eqref{v-integral-eqn}.
\end{lem}

\begin{lem}\label{g-lambda-compare-lemma}
Let $\beta\ge\beta_1$ and $\lambda_2>\lambda_1>0$. Then 
\begin{equation*}
g_{\lambda_2}(r)<g_{\lambda_1}(r)\quad\forall r>0.
\end{equation*} 
\end{lem}
\begin{proof}
We will use a modification of the proof of Lemma 2.3 of \cite{HuK} to proof the lemma. Since $\lambda_2>\lambda_1$, by \eqref{blow-up-rate-x=0}
there exists a constant $r_1>0$ such that $g_{\lambda_2}(r)<g_{\lambda_1}(r)$ for any $0<r\le r_1$. Let $(0,r_0)$ be the maximal interval such that
\begin{equation}\label{g-lambda-compare-interval}
g_{\lambda_2}(r)<g_{\lambda_1}(r)\quad\forall 0<r<r_0.
\end{equation}
Suppose $r_0<\infty$. Then $g_{\lambda_2}(r_0)=g_{\lambda_1}(r_0)$ and $g_{\lambda_2}'(r_0)\ge g_{\lambda_1}'(r_0)$. If $\beta>\beta_1$, then by Lemma \ref{g-integral-eqn-lemma} both $g_{\lambda_1}$ and $g_{\lambda_2}$ satisfy \eqref{v-integral-eqn}. Hence if $\beta>\beta_1$, by \eqref{beta>beta1-cond}, \eqref{g-lambda-compare-interval}, and Lemma \ref{g-integral-eqn-lemma}, 
\begin{align*}
r_0^{n-1}(g_{\lambda_2}^m)'(r_0)
=&-\beta r_0^ng_{\lambda_2}(r_0)+(n\beta-\alpha)\int_0^{r_0}g_{\lambda_2}(\rho)\rho^{n-1}\,d\rho\\
<&-\beta r_0^ng_{\lambda_1}(r_0)+(n\beta-\alpha)\int_0^{r_0}g_{\lambda_1}(\rho)\rho^{n-1}\,d\rho\\
=&r_0^{n-1}(g_{\lambda_1}^m)'(r_0).
\end{align*}
If $\beta=\beta_1$, by Lemma \ref{g-integral-eqn-lemma} both $g_{\lambda_1}$ and $g_{\lambda_2}$ satisfy \eqref{v-integral-eqn2}. Hence by \eqref{g-lambda-compare-interval}, 
\begin{equation*}
r_0^{n-1}(g_{\lambda_2}^m)'(r_0)=-r_0^ng_{\lambda_2}(r_0)+\beta\lambda_2^{-\frac{\rho_1}{(1-m)\beta}}
<-r_0^ng_{\lambda_1}(r_0)+\beta\lambda_1^{-\frac{\rho_1}{(1-m)\beta}}=r_0^{n-1}(g_{\lambda_1}^m)'(r_0).
\end{equation*}
Hence for any $\beta\ge\beta_1$,
\begin{align*}
&(g_{\lambda_2}^m)'(r_0)<(g_{\lambda_1}^m)'(r_0)\notag\\
\Rightarrow\quad&mg_{\lambda_2}^{m-1}(r_0)g_{\lambda_2}'(r_0)<mg_{\lambda_1}^{m-1}(r_0)g_{\lambda_2}'(r_0)\notag\\
\Rightarrow\quad&g_{\lambda_2}'(r_0)<g_{\lambda_1}'(r_0)
\end{align*} 
and contradiction arises. Thus $r_0=\infty$ and the lemma follows.
\end{proof}

By direct computation $\mathcal{C}(r)$ satisfies \eqref{v-integral-eqn} and \eqref{ode}. 
Since $\mathcal{C}(r)\to\infty$ as $r\to 0$ and $r^{\frac{\alpha}{\beta}}\mathcal{C}(r)\to 0$ as $r\to 0$, by Lemma \ref{g-integral-eqn-lemma}, Lemma \ref{v-integral-eqn-lem}, and an argument similar to the proof of Lemma \ref{g-lambda-compare-lemma}
we have the following result.

\begin{lem}\label{comparison-lemma}
Let $\beta\ge\beta_1$. Then 
\begin{equation*}\label{v-12-compare}
v_{\lambda_1}(r)<v_{\lambda_2}(r)\quad\forall r\ge 0, \lambda_2>\lambda_1>0
\end{equation*} 
and 
\begin{equation}\label{v-g-lambda-compare}
v_{\lambda}(r)<\mathcal{C}(r)<g_{\lambda}(r)\quad\forall r>0, \lambda>0.
\end{equation} 
\end{lem}

\begin{thm}\label{uniqueness-thm}
Let $\beta\ge\beta_1$ and $\lambda>0$. Suppose $g_1$ and $g_2$ are two radially symmetric solutions of \eqref{elliptic-eqn} in
$\R^n\setminus\{0\}$ that satisfies
\eqref{blow-up-rate-x=0}. Then $g_1=g_2$ on $\R^n\setminus\{0\}$.
\end{thm}
\begin{proof}
We choose a monotone decreasing sequence $\lambda_i>1$ for all $i\in\Z^+$ such that $\lambda_i\to 1$ as $i\to\infty$. Let
\begin{equation*}
w_i(r)=\lambda_i^{\frac{2}{1-m}}g_1(\lambda_i r)
\end{equation*}
Then $w_i$ satisfies \eqref{elliptic-eqn} in $\R^n\setminus\{0\}$ and
\begin{equation}\label{w-i-monotonicity}
\lim_{r\to 0}r^{\frac{\alpha}{\beta}}w_i(r)=\lambda_i^{-\frac{\rho_1}{(1-m)\beta}}\lim_{r\to 0}(\lambda_i r)^{\frac{\alpha}{\beta}}g_i(\lambda_i r)
=(\lambda_i\lambda)^{-\frac{\rho_1}{(1-m)\beta}}<(\lambda_{i+1}\lambda)^{-\frac{\rho_1}{(1-m)\beta}}<\lambda^{-\frac{\rho_1}{(1-m)\beta}}\quad\forall i\in\Z^+.
\end{equation}
By \eqref{w-i-monotonicity} and Lemma \ref{g-lambda-compare-lemma}, 
\begin{align*}
&w_i(r)<w_{i+1}(r)<g_2(r)\quad\forall r>0, i\in\Z^+\\
\Rightarrow\quad&g_1(r)\le g_2(r)\qquad\qquad\quad\forall r>0\quad\mbox{ as }i\to\infty.
\end{align*}
Similarly by interchanging the role of $g_1$ and $g_2$ in the above argument we get $g_1(r)\ge g_2(r)$ for all $r>0$. Hence $g_1=g_2$ on $\R^n\setminus\{0\}$ and the theorem follows.
\end{proof}

We are now ready to prove Theorem \ref{blow-up-self-similar-soln-thm}.

\noindent{\ni{\it Proof of Theorem \ref{blow-up-self-similar-soln-thm}:}}
By Theorem \ref{uniqueness-thm} we only need to prove existence of radially symmetric solution of \eqref{elliptic-eqn} in $\R^n\setminus\{0\}$ or solution of \eqref{ode} that satisfies \eqref{blow-up-rate-x=0}, \eqref{g-lambda'-negative} and \eqref{g-lower-bd}, when $\beta\ge\frac{m\rho_1}{n-2-nm}$. Let $\beta\ge\frac{m\rho_1}{n-2-nm}$.

\noindent{\bf Claim 1}: For any $\xi_0>0$, there exists a radially symmetric solution 
$g$ of \eqref{elliptic-eqn} in $\R^n\setminus B_{\xi_0}$ which satisfies 
\begin{equation}\label{initial-condition-at-xi}
g(\xi_0)=\lambda^{-\frac{\rho_1}{(1-m)\beta}}\xi_0^{-\frac{\alpha}{\beta}}\quad\mbox{ and }\quad
g'(\xi_0)=-\frac{\alpha}{\beta}\lambda^{-\frac{\rho_1}{(1-m)\beta}}\xi_0^{-\frac{\alpha}{\beta}-1}.
\end{equation}
In order to prove this claim we first observe that by the standard O.D.E. theory there exist $\3>0$ and a solution $g$ of \eqref{ode} in $(\xi_0,\xi_0+\3)$ which satisfies \eqref{initial-condition-at-xi}. Let $(\xi_0,R_0)$ be the maximal interval of existence of solution of \eqref{ode} which satisfies \eqref{initial-condition-at-xi}.
Let $\2{w}(r)=r^{\frac{\alpha}{\beta}}g(r)$,
$h_1(r)=g(r)+(\beta/\alpha)rg'(r)$, and 
\begin{equation*}
f_1(r)=g(r)^{m-1}\mbox{exp}\,\left(\frac{\beta}{m}\int_{\xi_0}^r\rho g(\rho)^{1-m}\,d\rho\right). 
\end{equation*}
By \eqref{initial-condition-at-xi}, $h_1(\xi_0)=0$. As observed in \cite{Hs1}, $h_1$ satisfies
\begin{equation*}\label{h1-eqn}
h_1'+\left(\frac{n-2-\frac{m\alpha}{\beta}}{r}-(1-m)\frac{g'}{g}
+\frac{\beta}{m}rg^{1-m}
\right)h_1=\frac{n-2-\frac{m\alpha}{\beta}}{r}g\ge 0
\end{equation*}
in $(\xi_0,R_0)$. Hence
\begin{align}
&(r^{n-2-\frac{m\alpha}{\beta}}f_1(r)h_1(r))'\ge 0\qquad\qquad\qquad\qquad\quad\forall \xi_0<r<R_0\notag\\
\Rightarrow\quad&r^{n-2-\frac{m\alpha}{\beta}}f_1(r)h_1(r)\ge \xi_0^{n-2-\frac{m\alpha}{\beta}}f_1(\xi_0)h_1(\xi_0)=0\quad\forall \xi_0<r<R_0\notag\\
\Rightarrow\quad&h_1(r)\ge 0\qquad\qquad\qquad\qquad\qquad\qquad\qquad\forall \xi_0<r<R_0\label{h1-positive}\\
\Rightarrow\quad&\2{w}'(r)=\frac{\alpha}{\beta}r^{\frac{\alpha}{\beta}-1}h_1(r)\ge 0\qquad\qquad\qquad\qquad\,\forall \xi_0<r<R_0\notag\\
\Rightarrow\quad&\2{w}(r)=r^{\frac{\alpha}{\beta}}g(r)\ge \2{w}(\xi_0)=\lambda^{-\frac{\rho_1}{(1-m)\beta}}\qquad\qquad\,\,\forall \xi_0<r<R_0.\label{w-positive}
\end{align}
By \eqref{ode}, \eqref{initial-condition-at-xi}, and \eqref{h1-positive},
\begin{align}\label{g'-negative}
&(r^{n-1}(g^m)')'=-\alpha r^{n-1}h_1(r)\le 0\qquad\forall \xi_0<r<R_0\notag\\
\Rightarrow\quad&r^{n-1}(g^m)'(r)\le\xi_0^{n-1}(g^m)'(\xi_0)<0\quad\,\forall \xi_0<r<R_0\notag\\
\Rightarrow\quad&g'(r)<0\qquad\qquad\qquad\qquad\qquad\,\,\,\forall \xi_0\le r<R_0.
\end{align}
By \eqref{w-positive} and \eqref{g'-negative},
\begin{equation}\label{g-upper-lower-bd}
\lambda^{-\frac{\rho_1}{(1-m)\beta}}r^{-\frac{\alpha}{\beta}}<g(r)<g(\xi_0)\quad\forall \xi_0<r<R_0.
\end{equation}
Suppose $R_0<\infty$. Since $g$ satisfies \eqref{g-lambda-integral-eqn} with $\xi=\xi_0$, $r\in (\xi_0,R_0)$, by \eqref{g-upper-lower-bd},
there exists a  constant $C>0$ independent of $R_0$ such that
\begin{equation}\label{g'-bd}
|g'(r)|\le C(1+R_0^n)\quad\forall \xi_0<r<R_0.
\end{equation} 
By \eqref{g-upper-lower-bd} and \eqref{g'-bd} we can extend $g$ to a solution of \eqref{ode} in $(\xi_0, R_1)$ that satisfies 
\eqref{initial-condition-at-xi}  for some $R_1>R_0$. This contradicts the choice of $R_0$. Hence $R_0=\infty$ and claim 1 follows.

By claim 1 for any $i\in\Z^+$ there exists a radially symmetric solution $g_i$ of \eqref{elliptic-eqn} in $\R^n\setminus B_{1/i}$ or equivalently a solution of \eqref{ode} in $(1/i,\infty)$ which satisfies 
\begin{equation}\label{gi-initial-condition}
g_i(1/i)=\lambda^{-\frac{\rho_1}{(1-m)\beta}}i^{\frac{\alpha}{\beta}}\quad\mbox{ and }\quad
g_i'(1/i)=-\frac{\alpha}{\beta}\lambda^{-\frac{\rho_1}{(1-m)\beta}}i^{\frac{\alpha}{\beta}+1}.
\end{equation}
Let $w_i(r)=r^{\frac{\alpha}{\beta}}g_i(r)$. By the proof of claim 1, 
\begin{equation}\label{gi'-negative}
g_i'(r)<0\quad\forall r>1/i, i\in\Z^+
\end{equation}
and
\begin{align}
&w_i'(1/i)=0\quad\mbox{ and }\quad w_i'(r)\ge 0\quad\forall r>1/i, i\in\Z^+\notag\\
\Rightarrow\quad&w_i(r)\ge w_i(1/i)=\lambda^{-\frac{\rho_1}{(1-m)\beta}}\qquad\quad\forall r>1/i, i\in\Z^+
\label{wi-lower-bd}\\
\Rightarrow\quad&g_i(r)\ge \lambda^{-\frac{\rho_1}{(1-m)\beta}}r^{-\frac{\alpha}{\beta}}\qquad\qquad\qquad\forall r>1/i, i\in\Z^+.
\label{gi-uniform-lower-bd}
\end{align}
By direct computation (cf. \cite{Hs1}),
\begin{equation}\label{qi-eqn}
\left(\frac{w_i'}{w_i}\right)'+\frac{n-1-\frac{2m\alpha}{\beta}}{r}\cdot\frac{w_i'}{w_i}
+m\left(\frac{w_i'}{w_i}\right)^2
+\frac{\beta r^{-1-\frac{\rho_1}{\beta}}w_i'}{mw_i^m}
=\frac{\alpha}{\beta}\cdot\frac{n-2-\frac{m\alpha}{\beta}}{r^2}\quad\forall r>1/i,i\in\Z^+.
\end{equation}
Let $s=\log r$ and $z_i=w_{i,s}/w_i$. Then $z_i(-\log i)=0$ and $z_i(s)\ge 0$ for any $s>-\log i$. By \eqref{qi-eqn} and a direct computation,
\begin{equation}\label{zi-eqn}
z_{i,s}+\left(n-2-\frac{2m\alpha}{\beta}\right)z_i+mz_i^2+\frac{\beta}{m}e^{-\frac{\rho_1}{\beta}s}w_i^{1-m}z_i=\frac{\alpha}{\beta}\left(n-2-\frac{m\alpha}{\beta}\right)\quad\forall s>-\log i,i\in\Z^+
\end{equation} 
By \eqref{wi-lower-bd} and \eqref{zi-eqn},
\begin{align}\label{zi-integral-ineqn}
&z_{i,s}+m(z_i^2+2C_1z_i)+\frac{\beta}{m}\lambda^{-\frac{\rho_1}{\beta}}e^{-\frac{\rho_1}{\beta}s}z_i\le C_2\quad\forall s>-\log i,i\in\Z^+\notag\\
\Rightarrow\quad&z_{i,s}+\frac{\beta}{m}\lambda^{-\frac{\rho_1}{\beta}}e^{-\frac{\rho_1}{\beta}s}z_i\le z_{i,s}+m(z_i+C_1)^2+\frac{\beta}{m}\lambda^{-\frac{\rho_1}{\beta}}e^{-\frac{\rho_1}{\beta}s}z_i\le C_2'\quad\forall s>-\log i,i\in\Z^+
\end{align}
where $C_1=\frac{1}{2m}\left(n-2-\frac{2m\alpha}{\beta}\right)$, $C_2=\frac{\alpha}{\beta}\left(n-2-\frac{m\alpha}{\beta}\right)>0$ and $C_2'=C_2+mC_1^2$.
Let 
\begin{equation*}
a(s)=-\frac{\beta^2}{m\rho_1}\lambda^{-\frac{\rho_1}{\beta}}e^{-\frac{\rho_1}{\beta}s}.
\end{equation*}
Then by \eqref{zi-integral-ineqn},
\begin{equation}\label{zi-integral-ineqn2}
z_i(s)\le C_2'\int_{-\log i}^se^{a(s')-a(s)}\,ds'\quad\forall s>-\log i,i\in\Z^+.
\end{equation} 
By the mean value theorem for any $s>s'>-\log i$ there exists a constant $s_1\in (s',s)$ such that
\begin{equation}\label{a-difference}
a(s')-a(s)=\frac{\beta}{m}\lambda^{-\frac{\rho_1}{\beta}}e^{-\frac{\rho_1}{\beta}s_1}(s'-s)
\le\frac{\beta}{m}\lambda^{-\frac{\rho_1}{\beta}}e^{-\frac{\rho_1}{\beta}s}(s'-s).
\end{equation}  
By  \eqref{zi-integral-ineqn2} and \eqref{a-difference}, 
\begin{align}\label{gi-uniform-upper-bd}
&z_i(s)\le C_2'\int_{-\log i}^s\mbox{exp}\,\left(\frac{\beta}{m}\lambda^{-\frac{\rho_1}{\beta}}e^{-\frac{\rho_1}{\beta}s}(s'-s)\right)\,ds'
\le C_0\lambda^{\frac{\rho_1}{\beta}}e^{\frac{\rho_1}{\beta}s}\quad\forall s>-\log i,i\in\Z^+\notag\\
\Rightarrow\quad&\frac{rw_{i,r}}{w_i}\le C_0\lambda^{\frac{\rho_1}{\beta}}r^{\frac{\rho_1}{\beta}}\qquad\qquad\qquad\qquad\forall r>1/i,i\in\Z^+\notag\\
\Rightarrow\quad&w_i(r)\le\lambda^{-\frac{\rho_1}{(1-m)\beta}}\mbox{exp}\,\left(\frac{\beta C_0}{\rho_1}\lambda^{\frac{\rho_1}{\beta}}r^{\frac{\rho_1}{\beta}}\right)\quad\forall r>1/i,i\in\Z^+\notag\\
\Rightarrow\quad&g_i(r)\le\lambda^{-\frac{\rho_1}{(1-m)\beta}}r^{-\frac{\alpha}{\beta}}\mbox{exp}\,\left(\frac{\beta C_0}{\rho_1}\lambda^{\frac{\rho_1}{\beta}}r^{\frac{\rho_1}{\beta}}\right)\quad\forall r>1/i,i\in\Z^+
\end{align}  
where $C_0=\frac{mC_2'}{\beta}$. By \eqref{gi-uniform-lower-bd} and \eqref{gi-uniform-upper-bd}, the equation
\eqref{elliptic-eqn} for the sequence $\{g_i\}_{i=1}^{\infty}$ is uniformly elliptic on every compact subset of $\R^n\setminus\{0\}$. Hence by standard Schauder's estimates \cite{GT} the sequence $\{g_i\}_{i=1}^{\infty}$ is uniformly continuous in $C^2(K)$ for any compact set $K\subset\R^n\setminus\{0\}$. By the Ascoli Theorem and a diagonalization argument the sequence $\{g_i\}_{i=1}^{\infty}$ has a subsequence which we may assume without loss of generality to be the sequence itself that converges uniformly in $C^2(K)$ for any compact set $K\subset\R^n\setminus\{0\}$ to some  function $g_{\lambda}\in C^2(\R^n\setminus\{0\})$ as $i\to\infty$. Then $g_{\lambda}$
is a radially symmetric solution of \eqref{elliptic-eqn} in $\R^n\setminus\{0\}$. Letting $i\to\infty$ in \eqref{gi-uniform-lower-bd} and \eqref{gi-uniform-upper-bd},
\begin{align}\label{g-lambda-upper-lower-bd}
&\lambda^{-\frac{\rho_1}{(1-m)\beta}}r^{-\frac{\alpha}{\beta}}\le g_{\lambda}(r)\le\lambda^{-\frac{\rho_1}{(1-m)\beta}}r^{-\frac{\alpha}{\beta}}\mbox{exp}\,\left(\frac{\beta C_0}{\rho_1}\lambda^{\frac{\rho_1}{\beta}}r^{\frac{\rho_1}{\beta}}\right)\quad\forall r>0,i\in\Z^+\notag\\
\Rightarrow\quad&\lambda^{-\frac{\rho_1}{(1-m)\beta}}\le r^{\frac{\alpha}{\beta}}g_{\lambda}(r)\le\lambda^{-\frac{\rho_1}{(1-m)\beta}}\mbox{exp}\,\left(\frac{\beta C_0}{\rho_1}\lambda^{\frac{\rho_1}{\beta}}r^{\frac{\rho_1}{\beta}}\right)\qquad\,\,\forall r>0,i\in\Z^+.
\end{align}
Letting $r\to 0$ in \eqref{gi'-negative} and \eqref{g-lambda-upper-lower-bd} we get \eqref{blow-up-rate-x=0} and \eqref{g-lambda'-negative} and the theorem follows.

{\hfill$\square$\vspace{6pt}}

\begin{cor}\label{g-lambda12-formula}
Let $n\ge 3$, $0<m<\frac{n-2}{n}$, $\rho_1>0$, $\lambda_1>0$, $\lambda_2>0$, $\beta\ge\beta_1$ and $\alpha=\frac{2\beta+\rho_1}{1-m}$. Let $g_{\lambda_1}$, $g_{\lambda_2}$, be two radially symmetric solution  of \eqref{elliptic-eqn} in $\R^n\setminus\{0\}$ that satisfies \eqref{blow-up-rate-x=0} with $\lambda$ being replaced by $\lambda_1$, $\lambda_2$ respectively. Then 
\begin{equation*}
g_{\lambda_2}(x)=(\lambda_2/\lambda_1)^{\frac{2}{1-m}}g_{\lambda_1}((\lambda_2/\lambda_1) x)\quad\forall x\in\R^n\setminus\{0\}.
\end{equation*}
\end{cor}
\begin{proof}
Let 
\begin{equation*}
\4{g}(x)=(\lambda_2/\lambda_1)^{\frac{2}{1-m}}g_{\lambda_1}((\lambda_2/\lambda_1) x)\quad\forall x\in\R^n\setminus\{0\}.
\end{equation*}
Then $\4{g}$ is a solution of \eqref{elliptic-eqn} in $\R^n\setminus\{0\}$ and
\begin{equation*}
\lim_{|x|\to 0}|x|^{\frac{\alpha}{\beta}}\4{g}(x)=(\lambda_2/\lambda_1)^{-\frac{\rho_1}{(1-m)\beta}}
\lim_{|x|\to 0}((\lambda_2/\lambda_1) |x|)^{\frac{\alpha}{\beta}}g_{\lambda_1}((\lambda_2/\lambda_1) x)=\lambda_2^{-\frac{\rho_1}{(1-m)\beta}}.
\end{equation*}
Hence by Theorem \ref{uniqueness-thm}, $\4{g}(x)\equiv g_{\lambda_2}(x)$ on $\R^n\setminus\{0\}$ and the corollary follows.
\end{proof}

By a similar argument we have the following corollary.

\begin{cor}\label{v-lambda-rescale-cor}
Let $n\ge 3$, $0<m<\frac{n-2}{n}$, $\rho_1>0$, $\lambda_1>0$, $\lambda_2>0$, $\beta\ge\beta_1$ and $\alpha=\frac{2\beta+\rho_1}{1-m}$. Let $v_{\lambda_1}$, $v_{\lambda_2}$, be two radially symmetric solution  of \eqref{elliptic-eqn} in $\R^n$ with $v_{\lambda_1}(0)=\lambda_1$, $v_{\lambda_2}(0)=\lambda_2$. Then 
\begin{equation}\label{v12-identity}
v_{\lambda_2}(x)=\frac{\lambda_2}{\lambda_1}\,v_{\lambda_1}((\lambda_2/\lambda_1)^{\frac{1-m}{2}}x)\quad\forall x\in\R^n.
\end{equation}
\end{cor}

Note that by an argument similar to the proof of \cite{Hs3},
\begin{equation}\label{g-lambda-infty-decay-rate}
\lim_{|x|\to\infty}|x|^2g_{\lambda}(x)^{1-m}=\frac{2m(n-2-nm)}{(1-m)\rho_1}.
\end{equation}

Then by \eqref{g-lambda-infty-decay-rate}, Lemma \ref{g-lambda-compare-lemma}, Corollary \ref{g-lambda12-formula} and an argument similar to the proof of Corollary 1.3 of \cite{Hs3} but with $g_{\lambda}$ replacing $v_{\lambda}$ in the proof there we get the following corollary.

\begin{cor}\label{v-rescale-limit-cor}
Let $n\ge 3$, $0<m<\frac{n-2}{n}$, $\lambda>0$, $\beta\ge\beta_1$ and $\alpha=\frac{2\beta+\rho_1}{1-m}$. Let $g_{\lambda}$ be the radially symmetric solution  of \eqref{elliptic-eqn} in $\R^n\setminus\{0\}$ that satisfies \eqref{blow-up-rate-x=0}. 
Then $g_{\lambda}(x)$ decreases and converges uniformly on $\R^n\setminus B_R$
to  $\mathcal{C}(x)$ for any $R>0$ as $\lambda\to\infty$.
\end{cor}

\section{Second order asymptotic of self-similar solutions}
\setcounter{equation}{0}
\setcounter{thm}{0}

In this section we will use a modification of the proof of \cite{DKS} to prove Theorem \ref{second-order-asymptotic-self-similar-soln-thm}. Let $s=\log r$ and
\begin{equation*}
\2{q}(s)=\left(r^{\frac{2}{1-m}}v_{\lambda}(r)\right)^m.
\end{equation*}
Then by the computation in section 3 of \cite{Hs1},
\begin{equation}\label{q-bar-eqn}
\2{q}_{ss}+\frac{n-2-(n+2)m}{1-m}\2{q}_s+\beta (\2{q}^{\frac{1}{m}})_s+\frac{\rho_1}{1-m}\2{q}^{\frac{1}{m}}-\frac{2m(n-2-nm)}{(1-m)^2}\2{q}=0\quad\mbox{ in }\R.
\end{equation}
Let $q(s)=\2{q}(s)/C_{\ast}^{\frac{m}{1-m}}$ where $C_{\ast}$ is given by \eqref{c-ast-defn}. Then by \eqref{q-bar-eqn},
\begin{equation}\label{q-eqn}
q_{ss}+\left(\frac{n-2-(n+2)m}{1-m}+\frac{\beta C_{\ast}}{m}q^{\frac{1}{m}-1}\right)q_s+\frac{2m(n-2-nm)}{(1-m)^2}\left(q^{\frac{1}{m}}-q\right)=0\quad\mbox{ in }\R.
\end{equation}
We now linearize \eqref{q-eqn} around the constant 1 solution by setting $q=1+w$ in \eqref{q-eqn}. Then $w$ satisfies
\begin{equation}\label{w-eqn}
w_{ss}+\left(\frac{n-2-(n+2)m}{1-m}+\frac{\beta C_{\ast}}{m}(1+w)^{\frac{1}{m}-1}\right)w_s+\frac{2m(n-2-nm)}{(1-m)^2}\left((1+w)^{\frac{1}{m}}-1-w\right)=0\quad\mbox{ in }\R
\end{equation}
and $w(s)>-1$ for all $s\in\R$.
Then the linearized operator of \eqref{w-eqn} around $w=0$ is
\begin{equation*}\label{linearized-operator}
Lw:=w_{ss}+\left(\frac{n-2-(n+2)m}{1-m}+\frac{2\beta(n-2-nm)}{(1-m)\rho_1}\right)w_s
+\frac{2(n-2-nm)}{(1-m)}w.
\end{equation*}
Note that the function $e^{-\gamma s}$ is a solution of $Lw=0$ if and only if $\gamma$ satisfies \eqref{gamma-eqn} whose two roots $\gamma_2>\gamma_1>0$ if $\beta>\beta_0$. 
We now rewrite \eqref{w-eqn} as
\begin{equation}\label{w-eqn2}
w_{ss}+\left(\frac{n-2-(n+2)m}{1-m}+\frac{2\beta(n-2-nm)}{(1-m)\rho_1}\right)w_s+\frac{2(n-2-nm)}{(1-m)}w=f
\end{equation}
where
\begin{equation*}
f(s)=-\frac{2m(n-2-nm)}{1-m}\left\{\frac{\beta}{m\rho_1}\left((1+w)^{\frac{1}{m}-1}-1\right)w_s+\frac{1}{1-m}\left((1+w)^{\frac{1}{m}}-1-\frac{1}{m}w\right)\right\}.
\end{equation*}
Let 
\begin{equation*}
\phi (z)=(1+z)^{\frac{1}{m}}-1-\frac{1}{m}z\quad\forall z>-1
\end{equation*}
and 
\begin{equation*}
\4{\phi}(s)=\phi (w(s)).
\end{equation*}
Then $\phi(z)$ is a non-negative convex function satisfying 
\begin{equation}\label{phi-lower-upper-bd}
a_1z^2\le \phi (z)\le a_2z^2\quad\forall |z|\le 1/10
\end{equation}
for some constants $a_2>a_1>0$ and
\begin{equation*}
f(s)=-\frac{2m(n-2-nm)}{1-m}\left\{\frac{\beta}{\rho_1}\4{\phi}'(s)+\frac{1}{1-m}\4{\phi}(s)\right\}.
\end{equation*}
Since $q(s)\to 0$ as $s\to -\infty$, $w(s)\to -1$ as $s\to -\infty$.
By the result of \cite{Hs3}, $q(s)\to 1$ as $s\to\infty$. 
Hence $w(s)\to 0$ as $s\to\infty$. Let $s_i=-i^2$ for all $i\in\Z^+$. Then by the mean value theorem for any $i\in\Z^+$ there exists a constant $s_i'\in (s_{i+1},s_i)$ such that
\begin{align}\label{w'-limit}
&|w_s(s_i')|=\left|\frac{w(s_i)-w(s_{i+1})}{s_i-s_{i+1}}\right|\le\frac{|w(s_i)|+|w(s_{i+1})|}{2i+1}\notag\\
\Rightarrow\quad&\lim_{i\to\infty}|w_s(s_i')|=0.
\end{align} 
Since $\gamma_1$, $\gamma_2$ are roots of \eqref{gamma-eqn},
\begin{equation}\label{sum-roots}
\gamma_1+\gamma_2=\left(\frac{n-2-(n+2)m}{1-m}+\frac{2\beta(n-2-nm)}{(1-m)\rho_1}\right).
\end{equation}
Multiplying \eqref{w-eqn2} by $e^{\gamma_1 s}$ and integrating over $(s_i',s)$, by \eqref{sum-roots} and integration by parts,
\begin{equation}\label{w-1st-order-eqn0}
\int_{s_i'}^se^{\gamma_1t}f(t)\,dt=e^{\gamma_1s}w_s(s)-e^{\gamma_1s_i'}w_s(s_i')+\gamma_2
(e^{\gamma_1s}w(s)-e^{\gamma_1s_i'}w(s_i')).
\end{equation}
Letting $i\to\infty$ in \eqref{w-1st-order-eqn0}, by \eqref{w'-limit},
\begin{equation}\label{w-1st-order-eqn}
w_s(s)+\gamma_2w(s)=e^{-\gamma_1s}\int_{-\infty}^se^{\gamma_1t}f(t)\,dt.
\end{equation} 
Similarly,
\begin{equation}\label{w-1st-order-eqn2}
w_s(s)+\gamma_1w(s)=e^{-\gamma_2s}\int_{-\infty}^se^{\gamma_2t}f(t)\,dt.
\end{equation} 
Subtracting \eqref{w-1st-order-eqn2} from \eqref{w-1st-order-eqn},
\begin{equation}\label{w-explicit-formula}
w(s)=\frac{1}{\gamma_2-\gamma_1}\left\{e^{-\gamma_1s}\int_{-\infty}^se^{\gamma_1t}f(t)\,dt-e^{-\gamma_2s}\int_{-\infty}^se^{\gamma_2t}f(t)\,dt\right\}\quad\forall s\in\R.
\end{equation} 
We are now ready to proof Theorem \ref{second-order-asymptotic-self-similar-soln-thm}.

\noindent{\ni{\it Proof of Theorem \ref{second-order-asymptotic-self-similar-soln-thm}:}} By \eqref{w-explicit-formula} and integration by parts,  
\begin{equation}\label{w-explicit-formula2}
w(s)=-M_0\left\{A_1(\beta)e^{-\gamma_1s}\int_{-\infty}^se^{\gamma_1t}\4{\phi}(t)\,dt-A_2(\beta)e^{-\gamma_2s}\int_{-\infty}^se^{\gamma_2t}\4{\phi}(t)\,dt\right\}\quad\forall s\in\R
\end{equation} 
where
\begin{equation}\label{C1-defn}
M_0=\frac{2m(n-2-nm)}{(1-m)(\gamma_2-\gamma_1)}
\end{equation}
and
\begin{equation}\label{Ai-beta-defn}
A_i(\beta)=\frac{1}{1-m}-\frac{\beta\gamma_i}{\rho_1},\quad i=1,2.
\end{equation}
Let $c_2=\left(1-\frac{m}{2}\right)^2$ and 
\begin{equation*}
b_0=\max\left(2\sqrt{\frac{2(1-m)}{(n-2-nm)(1-c_2^2)}},\frac{\sqrt{2}}{\sqrt{n-2-nm}}\right).
\end{equation*} 
If $0<m\le\frac{n-2}{n+2}$, we choose $a_0=b_0$.
If $\frac{n-2}{n+2}<m<\frac{n-2}{n}$, we will choose $a_0>0$ later such that it is strictly greater than $b_0$. Let 
\begin{equation*}
\beta_2=\max \left(a_0\rho_1,\beta_0,\beta_1\right)
\end{equation*} 
and $\beta>\beta_2$.
Then
\begin{equation*}
\sqrt{A(\beta)^2-8(n-2-nm)(1-m)}\ge c_2A(\beta).
\end{equation*}
Hence
\begin{align}
A_1(\beta)=&\frac{1}{1-m}\left\{1+\frac{\beta}{2\rho_1}\left(\sqrt{A(\beta)^2-8(n-2-nm)(1-m)}-A(\beta)\right)\right\}\notag\\
=&\frac{1}{1-m}\left\{1-\frac{4\beta(n-2-nm)(1-m)}{\rho_1\left(\sqrt{A(\beta)^2-8(n-2-nm)(1-m)}+A(\beta)\right)}\right\}\notag\\
\ge&\frac{1}{1-m}\left\{1-\frac{4\beta(n-2-nm)(1-m)}{\rho_1(1+c_2)A(\beta)}\right\}\label{A1-lower-bd2}\\
\ge&\frac{1}{1-m}\left(1-\frac{2(1-m)}{1+(1-(m/2))^2}\right)\qquad\quad\,\mbox{ if }\,0<m\le\frac{n-2}{n+2}
\notag\\
>&0.\qquad\qquad\qquad\qquad\qquad\qquad\qquad\quad\mbox{ if }\,0<m\le\frac{n-2}{n+2}\label{A1-positive}
\end{align}
Letting $\beta\to\infty$ in \eqref{A1-lower-bd2},
\begin{equation*}
\liminf_{\beta\to\infty}A_1(\beta)\ge\frac{1}{1-m}\left\{1-\frac{2(1-m)}{1+c_2}\right\}>0.
\end{equation*}
Hence for $\frac{n-2}{n+2}<m<\frac{n-2}{n}$, $n\ge 3$, we can choose  
$a_0>b_0$ such that 
\begin{equation}\label{A1-positive2}
A_1(\beta)>0\quad\forall \beta>\beta_2.
\end{equation}
Similarly,
\begin{align}\label{A2-negative}
A_2(\beta)=&\frac{1}{1-m}\left\{1-\frac{\beta}{2\rho_1}\left(\sqrt{A(\beta)^2-8(n-2-nm)(1-m)}+A(\beta)\right)\right\}\notag\\
\le&\frac{1}{1-m}\left\{1-\frac{\beta A(\beta)}{2\rho_1}\right\}\notag\\
\le&\frac{1}{1-m}\left\{1-\frac{\beta^2}{2\rho_1^2}(n-2-nm)\right\}\notag\\
<&0\qquad\qquad\qquad\qquad\qquad\qquad\forall \beta>\beta_2.
\end{align}
Since $e^{-\gamma_1(s-t)}\ge e^{-\gamma_2(s-t)}$ for all $t\in (-\infty,s)$, by \eqref{w-explicit-formula2},
\eqref{A1-positive}, \eqref{A1-positive2} and \eqref{A2-negative},
\begin{equation}\label{w-bd}
0>w(s)>-M_0(A_1(\beta)+|A_2(\beta)|)e^{-\gamma_1s}\int_{-\infty}^se^{\gamma_1 t}\4{\phi}(t)\,dt.
\end{equation}
By \eqref{w-bd} and an argument similar to the proof of Lemma 3.2 of \cite{DKS},
\begin{equation}\label{phi-tilde-L1-bd}
\int_{-\infty}^{\infty}e^{\gamma_1t}\4{\phi}(t)\,dt<\infty.
\end{equation}
Then by \eqref{w-explicit-formula2},  \eqref{phi-tilde-L1-bd}, and the same argument as the proof of Lemma 3.3 of \cite{DKS} we have
\begin{equation*}\label{w-decay-rate}
|w_1(s)|=-M_0A_1(\beta)I_1e^{-\gamma_1s}(1+o(1))\quad\mbox{ as }s\to\infty
\end{equation*}
where
\begin{equation*}
I_1=\int_{-\infty}^{\infty}e^{\gamma_1t}\4{\phi}(t)\,dt 
\end{equation*}
and the theorem follows.

{\hfill$\square$\vspace{6pt}}

\begin{cor}\label{B1-B2-ratio-cor}
Let $n\ge 3$, $0<m<\frac{n-2}{n}$, $\rho_1>0$, $\lambda_1>0$, $\lambda_2>0$, $\alpha=\frac{2\beta+\rho_1}{1-m}$ and $\beta>\beta_1$. Let $v_{\lambda_1}$, $v_{\lambda_2}$, be two radially symmetric solution  of \eqref{elliptic-eqn} in $\R^n$ with $v_{\lambda_1}(0)=\lambda_1$, $v_{\lambda_2}(0)=\lambda_2$ which satisfies \eqref{v-lambda-infty-behaviour} with $\lambda=\lambda_1, \lambda_2$, and  $B=B_{\lambda_1}, B_{\lambda_2}$,  respectively. Then
\begin{equation}\label{B12-relatiion}
B_{\lambda_2}=(\lambda_1/\lambda_2)^{\frac{(1-m)\gamma_1}{2}}B_{\lambda_1}.
\end{equation}
\end{cor}
\begin{proof}
By Corollary \ref{v-lambda-rescale-cor}, \eqref{v12-identity} holds. By Theorem \ref{second-order-asymptotic-self-similar-soln-thm},
\begin{equation}\label{vi-infty-behaviour}
v_{\lambda_i}(x)=\left(\frac{C_{\ast}}{|x|^2}\right)^{\frac{1}{1-m}}\left(1-B_{\lambda_i}|x|^{-\gamma_1}+o(|x|^{-\gamma_1})\right)\quad\forall i=1,2\quad\mbox{ as }|x|\to\infty.
\end{equation} 
Hence by \eqref{v12-identity} and \eqref{vi-infty-behaviour},
\begin{align}\label{vi-infty-behaviour2}
v_{\lambda_2}(x)=&\frac{\lambda_2}{\lambda_1}\left(\frac{C_{\ast}}{((\lambda_2/\lambda_1)^{\frac{1-m}{2}}|x|)^2}\right)^{\frac{1}{1-m}}\left(1-B_{\lambda_1}((\lambda_2/\lambda_1)^{\frac{1-m}{2}}|x|)^{-\gamma_1}+o(|x|^{-\gamma_1})\right)\quad\mbox{ as }|x|\to\infty\notag\\
=&\left(\frac{C_{\ast}}{|x|^2}\right)^{\frac{1}{1-m}}\left(1-B_{\lambda_1}((\lambda_2/\lambda_1)^{\frac{1-m}{2}}|x|)^{-\gamma_1}+o(|x|^{-\gamma_1})\right)\qquad\qquad\qquad\quad\mbox{ as }|x|\to\infty.
\end{align} 
By \eqref{vi-infty-behaviour} and \eqref{vi-infty-behaviour2}, we get \eqref{B12-relatiion} and the corollary follows.
\end{proof}

\section{Second order asymptotic of blow-up solutions}
\setcounter{equation}{0}
\setcounter{thm}{0}

In this section we will prove Theorem \ref{second-order-asymptotic-blow-up-elliptic-soln-thm}.

{\ni{\it Proof of Theorem \ref{second-order-asymptotic-blow-up-elliptic-soln-thm}}:} Similar to section 3 we let
\begin{equation*}\label{w-g-lambda-defn}
w(s)=[(r^2/C_{\ast})^{\frac{1}{1-m}}g_{\lambda}(r)]^m-1, \quad r=e^s,\quad s\in\R.
\end{equation*}
Then $w$ satisfies \eqref{w-eqn2} in $\R$. By the variation of parameter formula for any $s_0\in\R$ there
exist constants $C_2(s_0)$, $C_3(s_0)$, such that  
\begin{equation}\label{w-formula2}
w(s)=\frac{1}{\gamma_2-\gamma_1}\left(e^{-\gamma_1 s}\int_{s_0}^se^{\gamma_1 t}f(t)\,dt-e^{-\gamma_2 s}\int_{s_0}^se^{\gamma_2 t}f(t)\,dt\right)+C_2(s_0)e^{-\gamma_1 s}-C_3(s_0)e^{-\gamma_2 s}
\end{equation}
holds for any $s\in\R$.
By Lemma \ref{comparison-lemma}, \eqref{v-g-lambda-compare} holds. Hence $w(s)>0$ for all $s\in\R$.
By \eqref{g-lambda-infty-decay-rate}, $w(s)\to 0$ as $s\to\infty$. By \eqref{w-formula2} and integration by parts,
\begin{align}\label{w-formula3}
w(s)=&-M_0\left\{A_1(\beta)e^{-\gamma_1s}\int_{s_0}^se^{\gamma_1t}\4{\phi}(t)\,dt-A_2(\beta)e^{-\gamma_2s}\int_{s_0}^se^{\gamma_2t}\4{\phi}(t)\,dt\right\}+C_2'(s_0)e^{-\gamma_1 s}-C_3'(s_0)e^{-\gamma_2 s}
\end{align} 
for any $s_0,s\in\R$ where $M_0$ and $A_1(\beta)$, $A_2(\beta)$, are given by \eqref{C1-defn} and \eqref{Ai-beta-defn} respectively and
\begin{equation*}
\left\{\begin{aligned}
&C_2'(s_0)=C_2(s_0)+\frac{M_0\beta}{\rho_1}\4{\phi}(s_0)e^{\gamma_1s_0}\\
&C_3'(s_0)=C_3(s_0)+\frac{M_0\beta}{\rho_1}\4{\phi}(s_0)e^{\gamma_2s_0}.
\end{aligned}\right.
\end{equation*} 
Let $\beta_2>0$ be as in the proof of Theorem \ref{second-order-asymptotic-self-similar-soln-thm} and $\beta>\beta_2$.
Then by the proof of Theorem \ref{second-order-asymptotic-self-similar-soln-thm},
\begin{align}
&A_1(\beta)>0>A_2(\beta)\label{A12-compare}\\
\Leftrightarrow\quad&\gamma_1<\frac{\rho_1}{(1-m)\beta}<\gamma_2.\label{gamma1-gamma2-growth-rate-compare}
\end{align}
Since $w(s)\to\infty$ as $r=e^s\to 0^+$,
\begin{align}\label{w-x-large-decay}
&\4{\phi}(s)\approx (1+w(s))^{\frac{1}{m}}\approx w(s)^{\frac{1}{m}}\approx \frac{r^{\frac{2}{1-m}}g_{\lambda}(r)}{C_{\ast}^{\frac{1}{1-m}}}\approx 
\frac{(\lambda e^s)^{-\frac{\rho_1}{(1-m)\beta}}}{C_{\ast}^{\frac{1}{1-m}}}
\quad\mbox{ as }r=e^s\to 0^+\notag\\
\Rightarrow\quad&\lim_{s\to -\infty}\4{\phi}(s)e^{\frac{\rho_1s}{(1-m)\beta}}=\lim_{s\to -\infty}w(s)^{\frac{1}{m}}e^{\frac{\rho_1s}{(1-m)\beta}}=\lambda^{-\frac{\rho_1}{(1-m)\beta}}C_{\ast}^{-\frac{1}{1-m}},
\end{align}
multiplying \eqref{w-formula3} by $e^{\gamma_2s}$ and letting $s\to -\infty$, by \eqref{A12-compare}, \eqref{gamma1-gamma2-growth-rate-compare}  and \eqref{w-x-large-decay} we get,
\begin{align}\label{c3'-eqn0}
&-M_0\left\{A_1(\beta)\lim_{s\to -\infty}e^{(\gamma_2-\gamma_1)s}
\int_{s_0}^se^{\gamma_1t}\4{\phi}(t)\,dt
-|A_2(\beta)|\int_{-\infty}^{s_0}e^{\gamma_2t}\4{\phi}(t)\,dt\right\}-C_3'(s_0)\notag\\
=&\lim_{s\to -\infty}e^{\left(\gamma_2-\frac{m\rho_1}{(1-m)\beta}\right)s}\cdot\lim_{s\to -\infty}e^{\frac{m\rho_1s}{(1-m)\beta}}w(s)\notag\\
=&0.
\end{align}
By \eqref{w-x-large-decay} there exist constants $C_4>0$, $C_5>0$ and $s_1<0$ such that
\begin{equation}\label{phi-lower-upper-decay}
C_4\le\4{\phi}(s)e^{\frac{\rho_1s}{(1-m)\beta}}\le C_5\quad\forall s\le s_1.
\end{equation}
Then by \eqref{gamma1-gamma2-growth-rate-compare} and \eqref{phi-lower-upper-decay},
\begin{align}\label{integral1-to-infty}
\left|\int_{s_0}^se^{\gamma_1t}\4{\phi}(t)\,dt\right|\ge&\frac{C_4\left\{e^{\left(\gamma_1-\frac{\rho_1}{(1-m)\beta}\right)s}-e^{\left(\gamma_1-\frac{\rho_1}{(1-m)\beta}\right)s_1}\right\}}{\frac{\rho_1}{(1-m)\beta}-\gamma_1}
-\left|\int_{s_0}^{s_1}e^{\gamma_1t}\4{\phi}(t)\,dt\right|\quad\forall s<s_1\notag\\
\to&\infty\quad\mbox{ as }s\to -\infty.
\end{align}
Hence by \eqref{gamma1-gamma2-growth-rate-compare}, \eqref{w-x-large-decay}, \eqref{integral1-to-infty} and the l'Hospital rule,
\begin{equation}\label{integral1-limit}
\lim_{s\to -\infty}\frac{\int_{s_0}^se^{\gamma_1t}\4{\phi}(t)\,dt}{e^{(\gamma_1-\gamma_2)s}}=\frac{1}{(\gamma_1-\gamma_2)}\cdot
\lim_{s\to -\infty}e^{\gamma_2s}\4{\phi}(s)=\frac{1}{(\gamma_1-\gamma_2)}\lim_{s\to -\infty}e^{\left(\gamma_2-\frac{\rho_1}{(1-m)\beta}\right)s}\cdot\lim_{s\to -\infty}e^{\frac{\rho_1s}{(1-m)\beta}}\4{\phi}(s)=0.
\end{equation}
By \eqref{c3'-eqn0} and \eqref{integral1-limit},
\begin{equation}\label{c3'-positive}
C_3'(s_0)=M_0|A_2(\beta)|\int_{-\infty}^{s_0}e^{\gamma_2t}\4{\phi}(t)\,dt>0\quad\forall s_0\in\R.
\end{equation}
Putting $s=s_0$ in \eqref{w-formula3},
\begin{equation}\label{c2'-positive}
C_2'(s_0)=e^{\gamma_1s_0}w(s_0)+e^{(\gamma_1-\gamma_2)s_0}C_3'(s_0)>0\quad\forall s_0\in\R.
\end{equation}
By \eqref{w-formula3}, \eqref{A12-compare} and \eqref{c3'-positive},
\begin{equation}\label{w-upper-bd2}
0<w(s)\le C_2'(s_0)e^{-\gamma_1 s}\quad\forall s>s_0.
\end{equation}
Since $w(s)\to 0$ as $s\to\infty$, there exists $s_1>0$ such that 
\begin{equation}\label{phi-tilde-bd}
0\le\4{\phi}(s)\le a_2w(s)^2\quad\forall s>s_1
\end{equation}
where the constant $a_2>0$ is as given in \eqref{phi-lower-upper-bd}. By \eqref{w-upper-bd2} and \eqref{phi-tilde-bd},
\begin{equation}\label{phi-tilde-exponential-decay}
0\le\4{\phi}(s)\le C_6e^{-2\gamma_1s}\quad\forall s>s_1
\end{equation}
where $C_6=a_2C_2'(0)^2$.
Multiplying \eqref{w-formula3} by $e^{\gamma_1s}$ and letting $s\to\infty$, by \eqref{A12-compare}, \eqref{gamma1-gamma2-growth-rate-compare}  and \eqref{w-x-large-decay},
\begin{align}\label{c3'-eqn}
&\lim_{s\to\infty}e^{\gamma_1s}w(s)\notag\\
=&-M_0\left\{A_1(\beta)\int_{s_0}^{\infty}e^{\gamma_1t}\4{\phi}(t)\,dt
+|A_2(\beta)|\lim_{s\to\infty}e^{(\gamma_1-\gamma_2)s}\int_{s_0}^se^{\gamma_2t}\4{\phi}(t)\,dt\right\}+C_2'(s_0)\quad\forall s_0\in\R.
\end{align}
By \eqref{phi-tilde-exponential-decay},
\begin{align}\label{integral3=0}
&0\le \int_{s_0}^se^{\gamma_2t}\4{\phi}(t)\,dt\le C_7(1+e^{(\gamma_2-2\gamma_1)s})\qquad\qquad\quad\forall s>\max(s_0,s_1), s_0\in\R\notag\\
\Rightarrow\quad&0\le e^{(\gamma_1-\gamma_2)s}\int_{s_0}^se^{\gamma_2t}\4{\phi}(t)\,dt\le C_7(e^{(\gamma_1-\gamma_2)s}+e^{-\gamma_1s})\quad\forall s>\max(s_0,s_1), s_0\in\R\notag\\
\Rightarrow\quad&\lim_{s\to\infty}e^{(\gamma_1-\gamma_2)s}\int_{s_0}^se^{\gamma_2t}\4{\phi}(t)\,dt=0\qquad\qquad\qquad\qquad\forall s_0\in\R
\end{align}
where $C_7>0$ is some constant depending on $s_0$ and $s_1$. 
By \eqref{c3'-positive}, \eqref{c2'-positive}, \eqref{c3'-eqn} and \eqref{integral3=0}, the limit
\begin{equation*}
B:=\lim_{s\to -\infty}e^{\gamma_1s}w(s)
\end{equation*}
exists and is given by 
\begin{equation}\label{b-eqn}
0\le B=M_0|A_2(\beta)|e^{(\gamma_1-\gamma_2)s}\int_{-\infty}^{s}e^{\gamma_2t}\4{\phi}(t)\,dt
-M_0A_1(\beta)\int_s^{\infty}e^{\gamma_1t}\4{\phi}(t)\,dt+e^{\gamma_1s}w(s)\quad\forall s\in\R.
\end{equation}
We claim that $B>0$. Suppose not. Then $B=0$. Hence by \eqref{b-eqn},
\begin{equation}\label{limit-ineqn1}
w(s)\le M_0A_1(\beta)e^{-\gamma_1s}\int_s^{\infty}e^{\gamma_1t}\4{\phi}(t)\,dt\quad\forall s\in\R.
\end{equation}
We now choose $\3\in \left(0,\frac{\gamma_1}{a_2M_0A_1(\beta)}\right)$ where $a_2$ is as given in \eqref{phi-lower-upper-bd} and we choose $s_2>s_1$ such that 
\begin{equation}\label{w-exponential-decay1}
e^{-\gamma_1s_2}\le\3\quad\mbox{ and }\quad e^{\gamma_1s}w(s)\le 1\quad\forall s\ge s_2.
\end{equation}
Let $a_3=\frac{a_2M_0A_1(\beta)\3}{\gamma_1}$. Then $0<a_3<1$. By \eqref{phi-tilde-bd}, \eqref{limit-ineqn1} and \eqref{w-exponential-decay1},
\begin{equation}\label{limit-ineqn2}
w(s)\le a_2M_0A_1(\beta)e^{-\gamma_1s}\int_s^{\infty}e^{-\gamma_1t}\,dt\le \frac{a_2M_0A_1(\beta)}{\gamma_1}e^{-2\gamma_1s}\le a_3e^{-\gamma_1s}\quad\forall s\ge s_2.
\end{equation}
By \eqref{phi-tilde-bd}, \eqref{limit-ineqn1}, \eqref{w-exponential-decay1} and \eqref{limit-ineqn2},
\begin{equation}\label{limit-ineqn3}
w(s)\le a_2M_0A_1(\beta)a_3^2e^{-\gamma_1s}\int_s^{\infty}e^{-\gamma_1t}\,dt\le \frac{a_2M_0A_1(\beta)}{\gamma_1}a_3^2e^{-2\gamma_1s}\le a_3^3e^{-\gamma_1s}\quad\forall s\ge s_2.
\end{equation}
Repeating the above argument we get that
\begin{align*}\label{limit-ineqn4}
&w(s)\le a_3^{2^{k+1}-1}\cdot e^{-\gamma_1s}\quad\forall s\ge s_2, k\in\Z^+\\
\Rightarrow\quad&w(s)\equiv 0\qquad\qquad\quad\forall s\ge s_2\quad\mbox{ as }\quad k\to\infty
\end{align*}
which contradicts the fact that $w(s)>0$ for all $s\in\R$. Hence $B>0$ and the theorem follows.

\hfill$\square$\vspace{6pt}

Note that when $m=\frac{n-2}{n+2}$, $n\ge 3$, $\rho_1=1$, $\beta>\beta_1=\frac{1}{2m}$, then by the result of \cite{DKS}
\eqref{A12-compare}  holds. Moreover
\begin{align*}
\gamma_1-\frac{1}{(1-m)\beta}=&\frac{\beta(n-2)-\sqrt{\beta^2(n-2)^2-4(n-2)}}{2}-\frac{1}{(1-m)\beta}\\
=&\frac{2(n-2)}{\beta(n-2)+\sqrt{\beta^2(n-2)^2-4(n-2)}}-\frac{n+2}{4\beta}\\
=&-\frac{(n+2)\sqrt{\beta^2(n-2)^2-4(n-2)}+(n-2)(n-6)\beta}{4\beta\left(\beta(n-2)+\sqrt{\beta^2(n-2)^2-4(n-2)}\right)}\\
<&0\quad\mbox{ if }\quad\beta>\frac{1}{2m}
\end{align*} 
and
\begin{align*}
\gamma_2-\frac{1}{(1-m)\beta}=&\frac{\beta(n-2)+\sqrt{\beta^2(n-2)^2-4(n-2)}}{2}-\frac{n+2}{4\beta}\\
=&\frac{2\beta^2(n-2)+2\beta\sqrt{\beta^2(n-2)^2-4(n-2)}-(n+2)}{4\beta}\\
>&0
\end{align*}
if
\begin{equation*}
4\beta^2[\beta^2(n-2)^2-4(n-2)]>[n+2-2\beta^2(n-2)]^2\quad
\Leftrightarrow\quad\beta>\beta_1=\frac{1}{2m}.
\end{equation*}
Hence \eqref{gamma1-gamma2-growth-rate-compare} holds when $m=\frac{n-2}{n+2}$, $n\ge 3$, $\rho_1=1$, $\beta>\beta_1=\frac{1}{2m}$.
Then our proof above gives another proof of the following result of \cite{DKS}.

\begin{thm}\label{second-order-asymptotic-self-similar-soln-thm0}
Let $n\ge 3$, $m=\frac{n-2}{n+2}$, $\rho_1=1$, $\lambda>0$, $\beta>\beta_1$, $\alpha=\frac{2\beta+1}{1-m}$. If $g_{\lambda}$ is a radially symmetric solution of \eqref{elliptic-eqn} in $\R^n\setminus\{0\}$ which satisfies \eqref{blow-up-rate-x=0}, then 
\eqref{g-lambda-infty-behaviour0} holds for some constants $B>0$ where $\gamma_1$ is given by \eqref{roots}.
\end{thm}

By Corollary \ref{g-lambda12-formula} and an argument similar to the proof of Corollary \ref{B1-B2-ratio-cor} we have the following result. 

\begin{cor}
Let $n\ge 3$, $0<m<\frac{n-2}{n}$, $\rho_1>0$, $\lambda_1>0$, $\lambda_2>0$, $\alpha=\frac{2\beta+\rho_1}{1-m}$ and $\beta>\beta_1$. Suppose $g_{\lambda_1}$, $g_{\lambda_2}$, are two radially symmetric solution  of \eqref{elliptic-eqn} in $\R^n\setminus\{0\}$ which satisfies \eqref{blow-up-rate-x=0} and \eqref{g-lambda-infty-behaviour0} with $\lambda=\lambda_1, \lambda_2$,
and $B=B_1, B_2$, respectively.  Then \eqref{B12-relatiion} holds.
\end{cor}

\section{Large time behaviour of solutions}
\setcounter{equation}{0}
\setcounter{thm}{0}

In this section we will prove Theorem \ref{convergence-thm1} and Theorem \ref{convergence-thm3}. Since the proof of Theorem \ref{convergence-thm1} and Theorem \ref{convergence-thm3} are similar to the proof of \cite{DS2} and \cite{HuK}, we will only sketch its proof here.
We first observe that by an argument similar to the proof of Corollary 2.2 of \cite{DS1} we have the following two lemmas.

\begin{lem}\label{L1-contraction-lem}
Let $n\ge 3$, $0<m<\frac{n-2}{n}$ and let $u_1$, $u_2$ be two solutions of \eqref{Cauchy-problem} in $\R^n\times (0,T)$ with initial values $u_{0,1}\ge 0$, $u_{0,2}\ge 0$, respectively. Suppose $u_{0,1}-u_{0,2}\in L^1(\R^n)$ and for any $0<T_1<T$ there exist constants $r_0>0$, $C>0$,  such that either $u_1(x,t)\ge C/|x|^{\frac{2}{1-m}}$ for all $|x|\ge r_0$, $0<t<T_1$, or $u_2(x,t)\ge C/|x|^{\frac{2}{1-m}}$ for all $|x|\ge r_0$, $0<t<T_1$ holds. Then 
\begin{equation*}
\int_{\R^n}|u_1(x,t)-u_2(x,t)|\,dx\le\int_{\R^n}|u_{0,1}-u_{0,2}|\,dx\quad\forall 0<t<T.
\end{equation*}
Hence if $\4{u}_1$, $\4{u}_2$, are the rescaled solution of $u_1$, $u_2$, given by \eqref{rescald-soln},
\begin{equation*}
\int_{\R^n}|\4{u}_1(x,s)-\4{u}_2(x,s)|\,dx\le e^{-(n\beta-\alpha)s}\int_{\R^n}|u_{0,1}-u_{0,2}|\,dx\quad\forall s>-\log T.
\end{equation*}
\end{lem}

\begin{lem}\label{L1-contraction-singular-soln-lem}
Let $n\ge 3$, $0<m<\frac{n-2}{n}$, $\rho_1>0$, $\beta>\frac{m\rho_1}{n-2-nm}$, $\alpha=\frac{2\beta+\rho_1}{1-m}$, and let $u_1$, $u_2$ be two solutions of \eqref{Cauchy-problem} in $(\R^n\setminus\{0\})\times (0,T)$ with initial values $u_{0,1}\ge 0$, $u_{0,2}\ge 0$, respectively and there exist constants $C_1>0$, $C_2>0$, such that
\begin{equation}
C_1\le |x|^{\frac{\alpha}{\beta}}u_i(x)\le C_2\quad\forall 0<|x|\le 1, x\in\R^n, i=1,2.
\end{equation}
Suppose $u_{0,1}-u_{0,2}\in L^1(\R^n\setminus\{0\})$ and for any $0<T_1<T$ there exist constants $r_0>0$, $C>0$,  such that either $u_1(x,t)\ge C/|x|^{\frac{2}{1-m}}$ for all $|x|\ge r_0$, $0<t<T_1$, or $u_2(x,t)\ge C/|x|^{\frac{2}{1-m}}$ for all $|x|\ge r_0$, $0<t<T_1$ holds. Then 
\begin{equation*}
\int_{\R^n\setminus\{0\}}|u_1(x,t)-u_2(x,t)|\,dx\le\int_{\R^n\setminus\{0\}}|u_{0,1}-u_{0,2}|\,dx\quad\forall 0<t<T.
\end{equation*}
Hence if $\4{u}_1$, $\4{u}_2$, are the rescaled solution of $u_1$, $u_2$, given by \eqref{rescald-soln} and $\beta>\beta_1$, then
\begin{equation*}
\int_{\R^n\setminus\{0\}}|\4{u}_1(x,s)-\4{u}_2(x,s)|\,dx\le e^{-(n\beta-\alpha)s}\int_{\R^n\setminus\{0\}}|u_{0,1}-u_{0,2}|\,dx\quad\forall s>-\log T.
\end{equation*}
\end{lem}

By \eqref{v-infty-behaviour}, \eqref{g-lambda-infty-decay-rate}, Lemma \ref{L1-contraction-lem}, Lemma \ref{L1-contraction-singular-soln-lem} and the same argument as the proof of Theorem 1.1 of \cite{HuK} we get Theorem \ref{convergence-thm3} and the following result.

\begin{thm}\label{convergence-thm2}
Let $n\ge 3$, $0<m<\frac{n-2}{n}$, $T>0$, $\rho_1=1$, $\beta>\beta_1$ and $\alpha=\frac{2\beta+1}{1-m}$. Let $\psi_{\lambda}$ be given by \eqref{self-similar-soln} and let $u_0$ satisfy 
\begin{equation*}\label{u0-bd-psi1-2}
\psi_{\lambda_1}(x,0)\le u_0(x)\le\psi_{\lambda_2}(x,0)\quad\mbox{ in }\R^n
\end{equation*}  
and 
\begin{equation}
u_0(x)-\psi_{\lambda_0}(x,0)\in L^1(\R^n)
\end{equation}
for some constants $\lambda_2>\lambda_1>0$ and $\lambda_0>0$. Let $u$ be  the maximal solution of \eqref{Cauchy-problem} and $\4{u}$ be given by \eqref{rescald-soln}.
Then the rescaled solution $\4{u}(\cdot,s)$ converges uniformly on every compact subset of $\R^n$ and in $L^1(\R^n)$ to $v_{\lambda_0}$ as $s\to\infty$. Moreover,
\begin{equation*}
\|\4{u}(\cdot,s)-v_{\lambda_0}\|_{L^1(\R^n)}\le e^{-(n\beta-\alpha)s}\|u_0-\psi_{\lambda_0}(\cdot,0)\|_{L^1(\R^n)}\quad\forall s>-\log T.
\end{equation*}
\end{thm}

By Theorem \ref{convergence-thm2} and and argument similar to the proof of Theorem 1.2 of \cite{HuK}, Theorem \ref{convergence-thm1} follows.

\end{document}